\documentclass[a4paper]{amsproc}
\usepackage{amssymb}
\usepackage{mathrsfs}
\usepackage{amsfonts}
\usepackage{amsthm}
\usepackage{amsmath,amscd}
\usepackage{amsxtra}     
\usepackage{bm}
\usepackage[all,cmtip]{xy}
\usepackage{epsfig}
\usepackage{verbatim}
\usepackage{color}
\usepackage{enumerate}
\usepackage{hyperref}
\usepackage{tikz}
\usepackage{bbold}
\usepackage{tikz-cd}
\usetikzlibrary{arrows,matrix,shapes,trees}
\usepackage[capitalise]{cleveref}

\theoremstyle{plain}
 \newtheorem{thm}{Theorem}[section]
 \newtheorem*{theorem*}{Theorem}
 \newtheorem{prop}{Proposition}[section]
 \newtheorem{lem}{Lemma}[section]
 \newtheorem{cor}{Corollary}[section]

 \newtheorem{lem'}{``Lemma''}[section]
 
 \newtheorem*{claim}{Claim}
\theoremstyle{definition}
 
 \newtheorem{defn}{Definition}[section]
\theoremstyle{remark}
 \newtheorem{rmk}{Remark}[section]
 \numberwithin{equation}{section}

\newcommand{\N}{{\mathbb N}}
\newcommand{\Q}{{\mathbb Q}}
\newcommand{\Z}{{\mathbb Z}}

\newcommand{\F}{{\mathbb F}}

\newcommand{\mr}{\mathrm}

\newcommand{\prism}{{\mathbb{\Delta}}}

\title[Log prismatic Dieudonn\'e theory]{Log prismatic Dieudonn\'e theory for log $p$-divisible groups over $\mathcal{O}_{K}$}

\subjclass[2010]{14L05 (primary), 14D06, 13D15 (secondary)}

\keywords{}

\author[Matti W\"urthen]{\bfseries Matti W\"urthen}
\author[Heer Zhao]{\bfseries Heer Zhao}

\address{
    Matti W\"urthen, 
    Institut f\"ur Mathematik , 
    Goethe-Universit\"at Frankfurt, 
    Frankfurt am Main 60325 
    Germany, 
    	wuerthen@math.uni-frankfurt.de}
    	
\address{
    Heer Zhao, 
    Fakult\"at f\"ur Mathematik, 
    Universit\"at Duisburg-Essen, 
    Essen 45117, 
    Germany, 
    	heer.zhao@uni-due.de}

\begin{document}

\vspace{18mm} \setcounter{page}{1} \thispagestyle{empty}

\begin{abstract}
Let $\mathcal{O}_{K}$ be a complete discrete valuation ring of mixed characteristic with perfect residue field, endowed with its canonical log-structure. We prove that log $p$-divisible groups over $\mathcal{O}_{K}$ correspond to Dieudonn\'e crystals on the absolute log-prismatic site of $\mathcal{O}_{K}$ endowed with the Kummer log-flat topology. The proof uses log-descent to reduce the problem to the classical prismatic correspondence, recently established by Ansch\"utz-Le Bras.
\end{abstract}

\maketitle

\section*{Introduction}
Let $\mathcal{O}_{K}$ be a complete discrete valuation ring of mixed characteristic $(0, p)$ with perfect residue field $k$. The problem of relating $p$-divisible groups over $\mathcal{O}_{K}$ to semi-linear algebraic objects has a long history. A complete classification (for $p>2$) was first given by Breuil \cite{Breuil-grp}, using crystalline coefficient objects over the Breuil ring $\mathcal{S}$, and later by Kisin \cite{Kisin-moduli}, using Breuil-Kisin modules over the Breuil-Kisin ring $\mathfrak{S}$. The latter approach was recently conceptualized and vastly generalized to more general base rings by Ansch\"utz-Le Bras (see \cite{ALB}), using the absolute prismatic site, introduced by Bhatt-Scholze.\\
Recall that on $Spec(\mathcal{O}_{K})$ there is a canonical log structure, induced from the special fiber. Log $p$-divisible groups over log-schemes were introduced by Kato. These objects arise as degenerations of $p$-divisible groups over $K$ (examples include the $p$-divisible groups arising from semistable abelian varieties). Log $p$-divisible groups are defined in analogy to the classical case as certain inductive limits $G=\varinjlim G_{n}$, where the $G_{n}$ are sheaves of abelian groups on the Kummer log-flat site, which Kummer log-flat locally become classical finite flat group schemes. Such finite group objects are not automatically representable by log schemes. Kato therefore singles out an important subcategory $(p-div/\mathcal{O}_{K})^{\mr{log}}_{d}\subset (p-div/\mathcal{O}_{K})^{\mr{log}}$, which consists of objects $G=\varinjlim G_{n}$, for which all $G_{n}$ as well as their Cartier-duals are representable by log-schemes. It is the category $(p-div/\mathcal{O}_{K})^{\mr{log}}_{d}$ which corresponds to semistable objects, whereas the full category $(p-div/\mathcal{O}_{K})^{\mr{log}}$ corresponds to objects that become semistable after passing to a tamely ramified extension of $K$.\\[0.5cm]

The log prismatic site over $\mathcal{O}_{K}$ has recently been introduced by Koshikawa and Koshikawa-Yao in \cite{Ko1} and \cite{Ko2}. It consists of so called log-prisms $(B, I, M_{B}, \delta_{log})$ over $\mathcal{O}_{K}$, i.e. bounded log-rings equipped with Frobenius structures on $B$ and $M_{B}$ respectively and endowed with a map $\mathcal{O}_{K}\to B/I$. The absolute log prismatic site $(\mathcal{O}_{K}, M_{\mathcal{O}_{K}})_{\prism}$ from \cite{Ko2} is then defined to be the site of all such log-prisms over $\mathcal{O}_{K}$ endowed with the strict $(p, I)$-completely flat topology. We remark here already that in fact our category is a bit different, since we only consider log-structures which are saturated. We will also consider an enhanced absolute log prismatic site $(\mathcal{O}_{K}, M_{\mathcal{O}_{K}})_{\prism, \mr{kfl}}$, where we also admit so called completely Kummer faithfully log-flat maps as coverings. There is a canonical projection of sites
\[\epsilon:(\mathcal{O}_{K}, M_{\mathcal{O}_{K}})_{\prism, \mr{kfl}}\to (\mathcal{O}_{K}, M_{\mathcal{O}_{K}})_{\prism}.\]
We may then define the category $DM((\mathcal{O}_{K}, M_{\mathcal{O}_{K}})_{\prism, \mr{kfl}})$ of finite locally free sheaves $\mathcal{M}$ over $\mathcal{O}_{\prism_{\mr{kfl}}}$, endowed with a Frobenius-linear map $\Phi_{\mathcal{M}}:\mathcal{M}\to \mathcal{M}$, such that the linearization $\phi^{*}\mathcal{M}\to \mathcal{M}$ has $\mathcal{I}$-torsion cokernel. Here $\mathcal{I}$ denotes the ideal sheaf $(B, I, M_{B}, \delta_{log})\mapsto I$ on the log-prismatic site. Note that we may and will also define the category $DM((\mathcal{O}_{K}, M_{\mathcal{O}_{K}})_{\prism})$ of log-Dieudonn\'e crystals on the site $(\mathcal{O}_{K}, M_{\mathcal{O}_{K}})_{\prism}$. It embeds fully faithfully into $DM((\mathcal{O}_{K}, M_{\mathcal{O}_{K}})_{\prism, \mr{kfl}})$, via pullback along the morphism of sites $\epsilon$. One reason for admitting more general coverings is that we want to use descent along the map $\mathcal{O}_{K}\to \mathcal{O}_{C}$, where $C:=\hat{\bar{K}}$ and $\mathcal{O}_{C}$ is equipped with its canonical log-structure. Since this log-structure is divisible (it is induced by a chart $\Q_{\geq 0}\to \mathcal{O}_{C}$), many things are simplified over $\mathcal{O}_{C}$. In particular, any log $p$-divisible group over $\mathcal{O}_{C}$ is in fact classical. Using log-descent, we then prove the following.
\begin{thm}
There is a an exact anti-equivalence of categories
\[\mathcal{M}_{\prism}^{\mr{log}}:(p-div/\mathcal{O}_{K})^{\mr{log}}\to DM((\mathcal{O}_{K}, M_{\mathcal{O}_{K}})_{\prism, \mr{kfl}}),\]
which restricts to an exact anti-equivalence of categories
\[\mathcal{M}_{\prism}^{\mr{log}}:(p-div/\mathcal{O}_{K})^{\mr{log}}_{d}\to DM((\mathcal{O}_{K}, M_{\mathcal{O}_{K}})_{\prism}).\]
\end{thm}
We remark that Kato has proved a similar classification theorem (see \cite[Theorem 6.3]{k-ld}), using log-crystalline coefficient objects, under certain conditions on the base scheme (for instance when $\mathcal{O}_{K}$ is totally unramified, but also for higher dimensional log-schemes).

\section*{Acknowledgements}
M.W. was supported by Deutsche Forschungsgemeinschaft (DFG, German Research Foundation), TRR 326 \textit{Geometry and Arithmetic of Uniformized Structures} , project number 444845124, and by Deutsche Forschungsgemeinschaft (DFG, German Research Foundation), Sachbeihilfe \textit{Vector bundles and local systems on non-Archimedean analytic spaces}, project number 446443754.

\section{The Kummer log-flat topology}
In this section we want to extend the definition of Kato's Kummer topologies to certain cases where the log-structure is induced by possibly non-finitely generated monoid, as this will often be the case in non-noetherian situations.
\begin{defn}
A monoid $P$ is called quasi-fine and saturated (qfs), if it is integral and saturated.\\
A morphism $\alpha:P\to Q$ of qfs monoids is called Kummer (or of Kummer type), if $\alpha$ is injective and for all $q\in Q$, there exists an $n\in \mathbb{N}$, such that $q^{n}\in \alpha(P)$.
\end{defn}
Any qfs monoid $P$ can be written as a filtered colimit of fs monoids. Indeed we have $P=\varinjlim\limits_{Q\subset P}Q^{sat}$, where $Q$ runs over all finitely generated submodules of $P$.\\
We will only work with saturated log-structures. A qfs log-ring is a pair $(A, M_{A})$, where $A$ is a commutative ring and $M_{A}$ is a sheaf of monoids on $Spec(A)_{\acute{e}t}$, satisfying the usual conditions and such that \'etale locally, $M_{A}$ is induced by a qfs monoid. If $(A, M_{A})\to (B, M_{B})$ and $(A, M_{A})\to (C, M_{C})$ are morphisms of log-rings, which are induced by morphisms of pre-log rings $(A, P_{A})\to (B, P_{B})$ and $(A, P_{A})\to (C, P_{C})$ respectively (i.e. $P_{A}, P_{B}, P_{C}$ are monoids and $P_{A}\to A$ is a chart for $M_{A}$, etc.). Then we define 
\[B\otimes_{A}^{log}C:=B\otimes_{A}C \otimes_{\Z[Q]}{\Z[Q^{sat}]}\]
where $Q:=P_{B}\oplus_{P_{A}}P_{C}$, and endow $B\otimes_{A}^{log}C$ with the log structure induced from the pre-log structure $Q\to B\otimes_{A}^{log}C$.

\begin{defn}
A morphism $f:(A, M_{A})\to (B, M_{B})$ of qfs log-rings is called Kummer, if locally there exists a Kummer morphism $\alpha:P\to Q$ of qfs monoids, such that $P\to M_{A}$ and $Q\to M_{B}$ are charts and the diagram
\begin{center}
$\xymatrix{
P\ar[r]^{\alpha}\ar[d] & Q\ar[d]\\
M_{A}\ar[r] & M_{B}
}$
\end{center}
commutes.\\
We call $f:(A, M_{A})\to (B, M_{B})$ log-flat, if after passing to a localization there exists a chart $\alpha:P\to Q$ of $f$, such that the induced map $A\otimes_{\Z[P]}\Z[Q]\to B$ is flat.\\
Let $(X, M_{X})$ be a log scheme. We call $X$ qfs, if \'etale locally $M_{X}$ is induced by a qfs monoid.\\
Kummer and log-flat morphisms of log-schemes are defined in analogy to the case of log-rings above.\\ 
For any qfs log-scheme $S$, we denote by $Log^{\mr{qfs}}(S)$ the category of qfs-log schemes over $S$.
\end{defn}
\begin{lem}\label{chart-form}
Let $(R, M_{R})\to (B, M_{B})$ be a Kummer log-flat morphism and let $P\to M_{R}$ be a chart. After a classically flat localization on $B$, we may assume that there exists a chart $Q\to B$ of $M_{B}$ and a chart $\alpha:P\to Q$ of $M_{R}\to M_{B}$, such that $\alpha$ is Kummer and $R\otimes_{\Z[P]}\Z[Q]\to B$ is flat.
\end{lem}
\begin{proof}
This is a non-finite type version of \cite[Lem. 1.11]{Kat2}. The proof in loc. cit. works also in this case. 
\end{proof}

The following is a generalization of \cite[Lem. 2.4]{Kat2}.
\begin{lem}\label{kat2.4}
Let $f:X\to Y$ be a Kummer morphism of qfs log schemes and let $Y'\to Y$ be a qfs log scheme over $Y$. Then 
\begin{enumerate}
\item $f':X'=X\times_{Y}Y'$ is Kummer.
\item If $f$ is moreover surjective, then $f'$ is surjective.
\end{enumerate}
\end{lem}
\begin{proof}
The proof in loc.cit works also in this case.
\end{proof}
For a qfs log-scheme $X$, the previous two lemmas imply that by taking qfs-Kummer log-flat coverings, we get a well-defined site $X_{kfl}$, which we call the Kummer flat site.\\
If $X$ is an fs-log scheme, usually, for example in \cite{Kat2}, one defines the Kummer log-flat topology, using only coverings by fs-log schemes which moreover are assumed to be locally of finite presentation. We will denote this site by $X_{\mr{fpkfl}}$. \\
For a monoid $P$ and $n\geq 0$, we denote by $P^{\frac{1}{n}}$ the $P$-monoid, for which $P^{\frac{1}{n}}\cong P$ and the structure map $P\to P^{\frac{1}{n}}$ identifies with $n:P\to P$. The following is a non finite type version of \cite[Prop 2.15]{niz1}.
\begin{prop}\label{kfl-loc}
Let $\Phi:(R, M_{R})\to (B, M_{B})$ be a Kummer log-flat morphism of log rings, and let $P\to M_{R}$ be a chart, such that $P^{\times}=\{1\}$. Then for any point $x\in Spec(B)$, there exists a commutative diagram
\begin{center}
$\xymatrix{
(P\to R)\ar[r]^{\phi}\ar[d] & (Q\to B)\ar[d]^{f}\\
\varinjlim\limits_{n}R\otimes_{\Z[P]}\Z[P^{\frac{1}{n}}]\ar[r]^{g} & (Q'\to C)
}$
\end{center}
of pre-log rings, where $\phi$ induces the map $\Phi$, and such that $f$ is log-flat, $g$ is classically flat and $x$ lies in the image of $Spec(C)\to Spec(B)$.
\end{prop}
\begin{proof}
Using \cref{chart-form} we may assume, after localizing (around $x$) on $B$, that there exists a chart $Q\to B$ of $M_{B}$ and a chart $\alpha:P\to Q$ of $M_{R}\to M_{B}$, such that $\alpha$ is Kummer, $R\otimes_{\Z[P]}\Z[Q]\to B$ is flat and $\Z[P]\to \Z[Q]$ is log-flat.
Since $P$ is torsion free and $Q$ is saturated, we may further assume that $Q$ is torsion free and thus also $P^{\mr{gp}}=Q^{\mr{gp}}$. But then the map $P\to \varinjlim\limits_{n} P^{\frac{1}{n}}$ factors as $P\to Q\to \varinjlim\limits_{n} P^{\frac{1}{n}}$. We can then define $C:=B\otimes_{R\otimes_{\Z[P]}\Z[Q]}(R\otimes_{\Z[P]}\varinjlim\limits_{n}\Z[P^{\frac{1}{n}}])$. Then $R\otimes_{\Z[P]}\varinjlim\limits_{n}\Z[P^{\frac{1}{n}}]\to C$ is flat, since it is the base change of the flat map $R\otimes_{\Z[P]}\Z[Q]\to B$. If we then endow $C$ with the log-structure induced by $Q':=\varinjlim\limits_{n}P^{\frac{1}{n}}\to C$, we obtain the desired commutative diagram.
\end{proof}

\begin{prop}\label{desc-lf}
Pushforward and pullback along the canonical morphism of ringed sites $X_{\mr{kfl}}\to X_{\mr{fpkfl}}$ yield quasi-inverse equivalences between the category of finite locally free sheaves  on $X_{\mr{kfl}}$ and the category of finite locally free sheaves on $X_{\mr{fpkfl}}$.
\end{prop}
\begin{proof}
We need to show that kfl-locally free sheaves are locally free for the fpkfl-topology. We may assume that $X=Spec(A)$ is affine and that $M_{X}$ is induced by a chart $P\to M_{X}$ with $P^{\times}=\{1\}$. Let $Y=Spec(B)\to X$ a Kummer log-flat cover and let $\mathcal{E}$ be a finite locally free $\mathcal{O}_{Y_{\mr{kfl}}}$-module. By \cref{kfl-loc}, after passing to a further kfl-covering of $Y$, we may assume that $\mathcal{E}$ is trivial and that $Y\times_{X}\tilde{X}\to \tilde{X}$ is classically flat, where $\tilde{X}$ again denotes $Spec(\varinjlim\limits_{n}A\otimes_{\Z[P]}\Z[P^{\frac{1}{n}}]$. Using classical flat descent for vector bundles, we are thus reduced to treating bundles that become trivial over the covering $\tilde{X}\to X$. Now the descent datum for such a bundle is given by an isomorphism 
\[(A\otimes (\Z[\tilde{P}]\otimes_{\Z[P]} \Z[\tilde{P}])\otimes \Z[(\tilde{P}\oplus_{P}\tilde{P})^{\mr{sat}}])^{d}\cong  (A\otimes (\Z[\tilde{P}]\otimes_{\Z[P]} \Z[\tilde{P}])\otimes \Z[(\tilde{P}\oplus_{P}\tilde{P})^{\mr{sat}}])^{d}.\]
Since $\Z[\tilde{P}]=\varinjlim\limits_{n}\Z[P^{\frac{1}{n}}]$, we see that the descent datum descends to $A\otimes_{\Z[P]}\Z[P^{\frac{1}{N}}]$, for some $N>0$.
\end{proof}

\begin{defn}
A $p$-adic formal log-scheme $(X, M_{X})$ is called quasi-fine, if locally in the \'etale topology, there exists a chart $P\to M_{X}$, such that $P$ is an integral quasi-coherent monoid.\\
We call $(X, M_{X})$ quasi-fine saturated (qfs), if locally there exists a chart $P$, such that $P$ is quasi-fine and saturated.\\
For a bounded qfs $p$-adic formal log-scheme $(S, M_{S})$, we denote by $Log^{qfs}(S)$ the category of bounded qfs $p$-adic formal log schemes over $S$.
\end{defn}

\begin{defn}
A morphism of qfs $p$-adically complete log rings $f:(A, M_{A})\to (B, M_{B})$ is said to be $p$-completely Kummer log-flat, if there exists a Kummer chart $P\to Q$ of $M_{A}\to M_{B}$, such that $\Z[P]$ and $\Z[Q]$ have bounded $p^{\infty}$-torsion and $A\hat{\otimes}_{\Z[P]}\Z[Q]\to B$ is classically $p$-completely flat.
\end{defn}

\begin{lem}\label{surjective}
Let $(M_{A}\to A)\to (M_{B}\to B)$ be a morphism of qfs $p$-adically complete pre-log rings, such that $Spf(B)\to Spf(A)$ is surjective and $M_{A}\to M_{B}$ is Kummer. Let further $(M_{A}\to A)\to (M_{C}\to C)$ be any other morphism of qfs pre-log rings. Then $Spf(B\hat{\otimes}_{A}^{log}C)\to Spf(C)$ is also surjective.
\end{lem}
\begin{proof}
This follows from \cref{kat2.4}.
\end{proof}
We finish the section by recalling some facts about log $p$-divisible groups.
\begin{defn}
Let $S$ be a $p$-adic formal log-scheme. The category $(fin/S)_{f}$ denotes the category of abelian sheaves $G$ on $S_{\mr{kfl}}$, for which there exists a Kummer log-flat covering $S'\to S$, such that $G\vert_{S'}$ is a classical finite locally free formal group scheme over $S'$. We further denote by $(fin/S)_{d}$ the full subcategory of $(fin/S)_{f}$, whose objects consists of those $G\in (fin/S)_{f}$, for which $G$ and its Cartier dual $G^{\vee}$ are representable by formal log-schemes over $S$.
\end{defn}
\begin{rmk}\label{rmk-fin}
The above definition is compatible with the one of Kato (where only the restricted Kummer log-flat topology is used). Namely, let $S=Spf(R)$ be affine and let $G \in (fin/S)_{f}$ be $p^{n}$-torsion, for some $n\geq 1$ so that $G\in (fin/Spec(R/p^{n}))_{f}$ and assume that the log structure of $S$ is fs. It then follows that $G$ is a finite Kummer log-flat group object in the sense of Kato (see \cite[\S 1]{k-ld}). Namely, using \cref{desc-lf}, one sees that $G$ becomes classical over a finitely presented (or indeed even a finite) fs Kummer log-flat covering of $S/p^{n}$. 
\end{rmk}
\begin{defn}
We define the category $(p-div/S)^{\mr{log}}$ of log $p$-divisible groups, whose objects are sheaves of abelian groups on $S_{\mr{kfl}}$, such that
\begin{enumerate}
\item $G=\varinjlim_{n}G_{n}, \textit{with } G_{n}=G[p^{n}]$.
\item The multiplication map $p:G\to G$ is surjective.
\item $G_{n}\in (fin/S)_{f}, \textit{for all } n > 0.$
\end{enumerate}
We further denote by $(p-div/S)^{\mr{log}}_{d}$ the full subcategory, consisting of objects $G$, for which $G_{n}\in  (fin/S)_{d}$, for all $n > 0$.
\end{defn}

\begin{rmk}
If $S=Spf(R)$ for a $p$-adically complete ring $R$, it is easy to see that the category $(p-div/S)^{\mr{log}}$ is equivalent to the category of log $p$-divisible groups over $Spec(R)$. In case the log structure on $S$ is fs, one also sees (as in \cref{rmk-fin}) that the category $(p-div/S)^{\mr{log}}$ agrees with the one defined by Kato in \cite[\S 4.1]{k-ld}. 
\end{rmk}
Log $p$-divisible groups satisfy effective descent for the Kummer log-flat topology.
\begin{lem}
The fibered category
\begin{align*}
&& (R, M_{R})  \longmapsto & (p-div/(Spf(R), M_{R}))^{\mr{log}} & 
\end{align*} 
over bounded $p$-complete log rings is a stack for the $p$-completely Kummer log-flat topology.
\end{lem}
\begin{proof}
It is enough to prove the statement for the finite level. So let $f:(A, M_{A})\to (B, M_{B})$ be a faithfully $p$-completely Kummer log-flat morphism of bounded $p$-complete log rings and let $G\in  (fin/(Spf(B), M_{B}))_{f}$, which is $p^{n}$-torsion, for some $n\geq 0$, and is equipped with a descent datum with respect to $f$. Since $(A, M_{A})$ and $(B, M_{B})$ are bounded, the morphism $f/p^{n}$ is faithfully Kummer log flat in the usual sense.
\end{proof}

\section{The log-prismatic site}
The log prismatic site was defined by Koshikawa in \cite{Ko1} (see also the recent preprint \cite{Ko2}). We refer to \cite[\S 2 and \S3]{Ko1} for the fundamental notions of $\delta_{log}$-rings and log-prisms. Note however that in contrast to \cite{Ko1}, we only work with saturated log structures. Moreover we will consider two different topologies on the log prismatic site (one being the classical flat and one being the Kummer log-flat topology).
\begin{defn}
Let $(X, M_{X})$ be a bounded $p$-adic formal qfs-log scheme. The category of log prisms over $X$ consists of qfs log-prisms $(B, M_{B},  J, \delta^{log})$, equipped with a morphism of formal log schemes $f:Spf(B/J)\to X$.\\
The (absolute) log prismatic site $(X, M_{X})_{\prism}$ of $X$ consists of the above category, equipped with the topology where coverings are given morphisms $(B, M_{B}, I, \delta_{B}^{log})\to (C, M_{C}, IC, \delta^{log}_{C})$, which are strict and faithfully $(p, IB)$-completely flat.\\
The category of log prisms over $X$ equipped with the topology where coverings are Kummer faithfully log-flat morphisms is called the (absolute) Kummer log-prismatic site. We denote this site by $(X, M_{X})_{\prism, \mr{kfl}}$.\\
If $X=Spf(R)$ is affine, we also write $(R, M_{R})_{\prism}$ (resp. $(R, M_{R})_{\prism, \mr{kfl}}$) instead of $(X, M_{X})_{\prism}$ (resp. $(X, M_{X})_{\prism, \mr{kfl}}$).
\end{defn}
\begin{rmk}\label{prism-proj}
Since any strict $(p, I)$-completely flat map is $(p, I)$-completely Kummer log-flat, there is a natural projection of sites $\epsilon_{\prism}:(X, M_{X})_{\prism, \mr{kfl}}\to (X, M_{X})_{\prism}$.
\end{rmk}

Let $R$ be a perfectoid ring, with a pseudo-uniformizer $\omega\in R$. Then the prism $(A_{inf}(R), (\xi))$ may be endowed with a log structure $M_{A_{inf}(R)}$, which is induced by the map of monoids $[-]:(R^{\flat}\cap (R^{\flat}[\frac{1}{\omega^{\flat}}])^{\times}\to A_{inf}(R)$, defined by sending an element to its Teichm\"uller lift. By setting $\delta_{log}=0$, we get a log-prism over $R$.
\begin{lem}\label{lem-perf-log}
Let $(A, I):=(A_{inf}(R), I)$ be a perfect prism, for some perfectoid ring $R$, endowed with the log structure $M_{A}$, defined above. Let $(B, J, M_{B})$ be a log prism, such that there exists a map $f:A/I\to B/J$ of log rings. Then there exists a unique morphism $(A, I, M_{A})\to (B, J, M_{B})$ of log prisms, which induces $f$.
\end{lem}
\begin{proof}
Set $R=A/I$ and fix a pseudouniformizer $p^{\flat}\in R^{\flat}$ in the tilt. Then $M_{A}$ is induced by the chart $[-]:Q=R^{\flat}\cap R^{\flat}[\frac{1}{p^{\flat}}]^{\times}\to A$.
Consider the morphism of pre-log rings $(0\to \Z_p)\to (Q\to A)$. We claim that the $p$-completed log-cotangent complex $\hat{L}_{(Q\to A)/(0\to \Z_p)}$ is trivial. The lemma then follows from deformation theory.\\
Then $\hat{L}_{Q\to A)/(0\to \Z_p)}\otimes^{L}_{\Z_p}\F_p=L_{(Q\to R^{\flat})/\F_p}$. The log structure is again induced by the monoid $Q$.\\
Now if $P^{\bullet}\to (R^{\flat}, M_{R^{\flat}})$ is the canonical free resolution, we see that the log structure of $P^{(i)}$, for any $i\geq 0$, admits a chart $Q^{(i)}$, such that $Q^{(i)gp}$ is $p$-divisible. The log differentials are defined as $\Omega^{1}_{(Q^{(i)}\to P^{(i)})/\F_{p}}:=\Omega^{1}_{P^{(i)}/\F_{p}}\oplus Q^{(i)^{gp}}\otimes_{\Z}P^{(i)}$ modulo the relations $(d\alpha^{i}(m), 0)=(0, \alpha^{i}(m)\otimes m)$, for any $m\in Q^{(i)}$, where $\alpha^{i}$ denotes the map $Q^{(i)}\to P^{(i)}$.\\
Now $R^{\flat}[\frac{1}{p^{\flat}}]^{\times}/(R^{\flat})^{\times}$ is $p$-divisible, so we see that $\Omega^{1}_{(Q^{(i)}\to P^{(i)})/\F_{p}}=\Omega^{1}_{P^{(i)}/\F_{p}}$. Thus, we get that the log cotangent complex coincides with the ordinary cotangent complex and we have $L_{(M_{R^{\flat}}\to R^{\flat})/\F_p}=L_{R^{\flat}/\F_{p}}=0$.
\end{proof}
From the lemma, we see that the log-prismatic site of a perfectoid ring has an initial object.
\begin{cor}\label{cor-perf}
Let $R$ be a perfectoid ring over $\Z_{p}$, endowed with the log structure $R[\frac{1}{p}]^{\times}\cap R\hookrightarrow R$. Then $(Spf(R), M_{R})_{\prism}$ has an initial object given by the log prism $(A_{inf}(R), (\xi), R^{\flat}[\frac{1}{p^{\flat}}]^{\times}\cap R^{\flat})$.
\end{cor}
\begin{proof}
Let $(B, J, M_{B}, \delta_{log})$ be a log-prism. Using the lemma above, we obtain a unique map $(A_{inf}(R), M_{A_{inf}(R)})\to (B, M_{B})$ of log-rings, which is also a map of prisms. One still needs to check that this is also compatible with the $\delta_{log}$-structures. This follows from the following.
\begin{claim}
Let $(B, I, M_{B}, \delta_{log})$ be a log prism and assume that there is a section $x\in M_{B}$, such that there exists a $p^{n}$-th root $x^{\frac{1}{p^{n}}}\in M_{B}$ for any $n\geq 1$. Then we have $\delta_{log}(x)=0$.
\end{claim}
Namely, by the definition of a $\delta_{log}$-structure, $\delta_{log}$ satisfies
\[\delta_{log}(xy)=\delta_{log}(x)+\delta_{log}(y)+p\delta_{log}(x)\delta_{log}(y)\]
for any $x, y\in M_{B}$. We then claim that for any $z\in M_{B}$, it follows that $\delta_{log}(z^{p})$ is divisible by $p\delta_{log}(z)$. Namely, we get
\begin{align*}
& \delta_{log}(z^{p}) & = & \delta_{log}(z)+\delta_{log}(z^{p-1})+p\delta_{log}(z)\delta_{log}(z^{p-1})\\
& & = & 2\delta_{log}(z)+\delta_{log}(z^{p-2})+p\delta_{log}(z)(\delta_{log}(z^{p-2})+\delta_{log}(z^{p-1})) \\
& & = & \cdots\\
& & = & p\delta_{log}(z)+p\delta_{log}(z)(\sum\limits_{k=1}^{p-1}\delta_{log}(z^{p-k}))
\end{align*}
If $x$ admits a $p^{n}$-th root for any $n\geq 1$, we thus see inductively that $\delta_{log}(x)$ is divisible by $p^{n}$ for all $n\geq 1$. But since $B$ is bounded, it is classically $p$-complete, so we get that $\delta_{log}(x)=0$.
\end{proof}
\section{The Kummer quasisyntomic site}
Let $P\to R$ be a pre-log ring over $\Z_p$. 
\begin{defn}
A morphism $f:(A, M_{A})\to (B, M_{B})$ of bounded $p$-complete qfs log rings is called Kummer log-quasisyntomic, if there exists a Kummer-type chart $\alpha:P\to Q$ for $f$, with $P$, $Q$ bounded monoids, such that the induced map $A\hat{\otimes}_{\Z[P]}\Z[Q]\to B$ is quasisyntomic.
\end{defn}
For any $p$-adically complete log formal scheme $X$, it again follows from \cref{surjective} that there exists a site, called the Kummer log-quasisyntomic site, where coverings are given by faithfully log-quasisyntomic maps. 

\begin{rmk}\label{rmk-qsyn}
In \cite[\S 3]{Ko2} a different notion of log-quasisyntomic maps has been introduced, using the log-cotangent complex. In general these two notions will differ - for example, if $(A, M_{A})\to (B, M_{B})$ is a log-quasisyntomic morphism in our sense, the underlying map $A\to B$ need not be flat. However, when defining the small Kummer log-quasisyntomic site over $\mathcal{O}_{K}$, these differences do not appear - i.e. in that case one may use either notion of log-quasisyntomic maps.
\end{rmk}

\begin{prop}
Let $\Phi:(R, M_{R})\to (B, M_{B})$ be a Kummer log-quasisyntomic morphism of $p$-complete log rings, and let $P\to M_{R}$ be a chart, such that $P^{\times}=\{1\}$. Then for any point $x\in Spf(B)$, there exists a commutative diagram
\begin{center}
$\xymatrix{
(P\to R)\ar[r]^{\phi}\ar[d] & (Q\to B)\ar[d]^{f}\\
(\varinjlim\limits_{n}R\otimes_{\Z[P]}\Z[P^{\frac{1}{n}}])_{p}^{\wedge}\ar[r]^{g} & (Q'\to C)
}$
\end{center}
of pre-log rings, where $\phi$ induces the map $\Phi$, and such that $f$ is Kummer log-quasisyntomic, $g$ is classically quasisyntomic and $x$ lies in the image of $Spf(C)\to Spf(B)$.
\end{prop}
\begin{proof}
This follows in the same way as \cref{kfl-loc}.
\end{proof}

\section{The log prismatic site over $\mathcal{O}_{K}$}
We will now specialize the discussion to the case of a complete discrete valuation ring $\mathcal{O}_{K}$ of mixed characteristic. Recall that we define $C:=\hat{\bar{K}}$.
\begin{lem}\label{strictness}
\begin{enumerate}
\item Any $p$-completely Kummer log-flat map $f:(\mathcal{O}_{C}, M_{\mathcal{O}_{C}})\to (R, M_{R})$ is strict, in particular classically $p$-completely flat.
\item Any Kummer log-flat map of prisms $(A_{inf}, (\xi), M_{A_{inf}})\to (A, \xi A,  M_{A})$ is strict.
\end{enumerate}
\end{lem}
\begin{proof}
Let $\alpha:\Q_{\geq 0}\hookrightarrow P$ be a chart for $f$, which is Kummer and where $P\to M_{R}$ is a chart for the log structure of $R$. Since $\Q_{\geq 0}$ is torsion-free and $P$ is saturated, we may also assume that $P$ is torsion free. Since $\alpha$ is Kummer, we then have $P=\Q=\Q_{\geq 0}^{\mr{gp}}$, hence $P=\Q_{\geq 0}$ or $P=\Q$, since $\Q$ contains no bigger proper submonoid than $\Q_{\geq 0}$. In both cases, this means that $M_{R}$ is induced by $\Q_{\geq 0}$, which means that $f$ is strict.\\
Part (2) follows in the same way.
\end{proof}

\begin{lem}\label{weakly-initial}
The log prism $(A_{inf}, (\xi), \mathcal{O}_{C}^{\flat}\backslash \{0\})$ is a weakly initial object of $(\mathcal{O}_{K}, M_{\mathcal{O}_{K}})_{\prism_{\mr{kfl}}}$.
\end{lem}
\begin{proof}
Let $(A, I, M_{A})$ be a log-prism over $\mathcal{O}_{K}$ and assume that there is a chart $P\to M_{A}$, and a chart $\N\to P$, for the map of log rings $(\mathcal{O}_{K}, M_{\mathcal{O}_{K}})\to (A/I, M_{A})$. Since $A_{inf}$ is a weakly initial object for the classical prismatic site, there is a flat covering $(A, I)\to (\tilde{A}, I\tilde{A})$, such that there exists a map of prisms $A_{inf}\to \tilde{A}$. We then endow $\tilde{A}/I$ with the log structure induced by $\tilde{P}=P\oplus_{\N}\Q_{\geq 0}\to \tilde{A}/I$, so in particular the map $\mathcal{O}_{C}\to \tilde{A}/I$, becomes a map of log-rings. Now first endow $\tilde{A}$ with the log structure induced from $P\to A\to \tilde{A}$. Then $(\tilde{A}, P)\to (\tilde{A}/I, \tilde{P})$ is a map of pre-log rings. Let then 
$(\tilde{A}', I', P')$ be the exactification of $(\tilde{A}, P)$ (see \cite[Constr. 2.17]{Ko1}), i.e. 
$(\tilde{A}', I', P')$ is a log prism with $\tilde{A}'/I'=\tilde{A}/I$ and the log structure on $\tilde{A}/I$ is induced by $P'$. Then by \cref{lem-perf-log}, the map $(\mathcal{O}_{C}, \Q_{\geq 0})\to (\tilde{A}/I, P')$ lifts to a unique map of log-prisms $A_{inf}\to \tilde{A}'$. By \cite[Prop. 2.16]{Ko1}, the saturation $\tilde{A}'\hat{\otimes}_{\Z[P']}\Z[(P')^{\mr{sat}}]$ may be endowed with a log prism structure and the map $A\to \tilde{A}'\hat{\otimes}_{\Z[P']}\Z[(P')^{\mr{sat}}]$ is a Kummer log-flat map of log-prisms. This finishes the proof.
\end{proof}

Recall from \cite[\S 4.1]{ALB} that there is a functor of topoi
$v_{*}:(\mathcal{O}_{K})_{\prism}^{\sim}\to (\mathcal{O}_{K})_{qsyn}^{\sim}$. This extends to the log case.

\begin{prop}\label{topoi-diagram}
There is a functor $\mu:(\mathcal{O}_{K}, M_{\mathcal{O}_{K}})_{\prism, \mr{kfl}}^{\sim}\to (\mathcal{O}_{K}, M_{\mathcal{O}_{K}})_{\mr{kqsyn}}^{\sim}$, which fits into a commutative diagram 
\begin{center}
$\xymatrix{
(\mathcal{O}_{K}, M_{\mathcal{O}_{K}})_{\prism, \mr{kfl}}^{\sim}\ar[r]\ar[d] & (\mathcal{O}_{K}, M_{\mathcal{O}_{K}})_{\mr{kqsyn}}^{\sim}\ar[d]\\
(\mathcal{O}_{K})_{\prism}^{\sim}\ar[r] & (\mathcal{O}_{K})_{qsyn}^{\sim}
}$
\end{center}
where the vertical arrows denote the projections from the log topoi to their classical counterparts and the lower functor is the functor $v_{*}$ constructed in \cite[\S 4.1]{ALB}.
\end{prop}

\begin{proof}
Consider the full subcategory $(\mathcal{O}_{K}, M_{\mathcal{O}_{K}})_{\prism. \mr{kfl}}^{\mr{s}}\subset (\mathcal{O}_{K}, M_{\mathcal{O}_{K}})_{\prism. \mr{kfl}}$, consisting of log prisms $(A, I, M_{A})$, such that the structure map $\mathcal{O}_{K}\to A/I$ is $p$-completely Kummer log-flat. We endow $ (\mathcal{O}_{K}, M_{\mathcal{O}_{K}})_{\prism. \mr{kfl}}^{\mr{s}}$ with the Grothendieck topology where again coverings are faithfully $(p, I)$-completely Kummer log-flat maps of log-prisms.  Consider the morphism of topoi $\nu_{*}:(\mathcal{O}_{K}, M_{\mathcal{O}_{K}})_{\prism. \mr{kfl}}^{\sim}\to (\mathcal{O}_{K}, M_{\mathcal{O}_{K}})_{\prism. \mr{kfl}}^{\mr{s}, \sim}$, defined by $\nu_{*}\mathcal{F}((A, I, M_{A})):=\mathcal{F}((A, I, M_{A}))$, for any $\mathcal{F}\in (\mathcal{O}_{K}, M_{\mathcal{O}_{K}})_{\prism. \mr{kfl}}^{\sim}$ and any $(A, I, M_{A})\in (\mathcal{O}_{K}, M_{\mathcal{O}_{K}})_{\prism. \mr{kfl}}^{\mr{s}}$.\\
Now consider the functor $u:(\mathcal{O}_{K}, M_{\mathcal{O}_{K}})_{\prism. \mr{kfl}}^{\mr{s}}\to (\mathcal{O}_{K}, M_{\mathcal{O}_{K}})_{\mr{kqsyn}}$, given by $(A, I, M_{A})\to A/I$. The following claim ensures that $u$ is cocontinuous.
\begin{claim}
Let $(A, I, M_{A})\in (\mathcal{O}_{K}, M_{\mathcal{O}_{K}})_{\prism, \mr{kfl}}^{\mr{s}}$ and let $(A/I, M_{A})\to (R, M_{R})$ be a Kummer log-quasisyntomic map of $p$-complete log rings. Then there exists a log prism $(B, IB, M_{B})$ which is $(p, I)$-completely faithfully flat over $(A, I, M_{A})$, such that $(A/I, M_{A})\to (B/IB, M_{B})$ factors through $(R, M_{R})$ and the map $(R, M_{R})\to (B/IB, M_{B})$ is $p$-completely faithfully flat.
\end{claim}
The proof of the claim proceeds just as the proof of \cref{weakly-initial}. We first remark that $\mathcal{O}_{K}\to A/I$ is classically $p$-completely flat. This follows, since it is completely Kummer log-flat, so in particular there is a Kummer type chart $\N\to P$ of $M_{\mathcal{O}_{K}}\to M_{A/I}$ and it is easy to see that $\Z[\N]\to \Z[P]$ must be flat. We then first apply \cite[Prop. 7.11]{BS-pris} to the quasisyntomic map $A/I\to \mathcal{O}_{C}\hat{\otimes}_{\mathcal{O}_{K}}A/I$, so there exists a prism $(\tilde{A}, I\tilde{A})$, such that $A\to \tilde{A}$ is completely flat. Now endow $\tilde{A}/I$ with the log structure induced by $\Q_{\geq 0}\oplus_{\N}P\to \mathcal{O}_{C}\hat{\otimes}_{\mathcal{O}_{K}}A/I\to \tilde{A}/I$ and $\tilde{A}$ with the log structure $P\to A\to \tilde{A}$. Let $\tilde{A}'$ be the saturation of the exactification of $\tilde{A}\to \tilde{A}/I$.\\
Now $R\hat{\otimes}^{\mr{log}}_{A/I}\tilde{A}'/I$ is $p$-completely Kummer log-flat over $\tilde{A}'/I$, which in turn is $p$-completely Kummer log-flat over $\mathcal{O}_{C}$. Now by \cref{strictness}, any $p$-completely Kummer log-flat log-ring over $\mathcal{O}_{C}$ is strict. We therefore may apply \cite[Prop. 7.11]{BS-pris} to $\tilde{A}'/I\to R\hat{\otimes}^{\mr{log}}_{A/I}\tilde{A}'/I$ to obtain a prism $(B, IB)$ over $\tilde{A}'$. Endowing $B$ with the log-structure induced from $\tilde{A}'$, yields the desired log-prism $(B, IB, M_{B})$.\\[0.5cm]
The cocontinuous functor $u$ induces a morphism of topoi $((\mathcal{O}_{K}, M_{\mathcal{O}_{K}})_{\prism. \mr{kfl}}^{\mr{s}, \sim}\to (\mathcal{O}_{K}, M_{\mathcal{O}_{K}})_{\mr{kqsyn}}^{\sim}$. Composing $u_{*}$ with the functor $\nu_{*}$ yields the desired functor 
\[(\mathcal{O}_{K}, M_{\mathcal{O}_{K}})_{\prism. \mr{kfl}}^{\sim}\to (\mathcal{O}_{K}, M_{\mathcal{O}_{K}})_{\mr{kqsyn}}^{\sim}.\]
\end{proof}

\begin{rmk}\label{left-adjoint}
The functor $\mu$ admits a left adjoint $\mu^{\natural}$.
Namely, first note that the functor $\nu_{*}:(\mathcal{O}_{K}, M_{\mathcal{O}_{K}})_{\prism. \mr{kfl}}^{\mr{s}}\to (\mathcal{O}_{K}, M_{\mathcal{O}_{K}})_{\prism, \mr{kfl}}^{\sim}$ admits a left adjoint $\nu^{\natural}$, defined by
\[\mu^{\natural}(\mathcal{F})((A, I, M_{A}))=\underset{(B, IB, M_{B})\in \mathcal{I}_{A})}{colim}\mathcal{F}((B, J, M_{B}))\]
where $\mathcal{I}_{A}$ is the category of pairs $((B, J, M_{B}), \phi)$, where $(B, J, M_{B})$ is a log prism, with $B/JB$ Kummer log-quasisyntomic over $\mathcal{O}_{K}$ and $\phi:(B, J, M_{B})\to (A, I, M_{A})$ is a map of log-prisms.\\
Then $\mu^{\natural}$ is defined as the composition of $\nu^{\natural}$ with $u^{*}$, where $(u_{*}, u^{*})$ denotes the morphism of topoi $((\mathcal{O}_{K}, M_{\mathcal{O}_{K}})_{\prism. \mr{kfl}}^{\mr{s}, \sim}\to (\mathcal{O}_{K}, M_{\mathcal{O}_{K}})_{\mr{kqsyn}}^{\sim}$, considered in the proof above.
\end{rmk}

\subsection{Products}
In this section we study the products
\begin{center}
$\mathcal{O}_{C}^{(n),log}=(\underbrace{\mathcal{O}_{C}\hat{\otimes}_{\mathcal{O}_{K}}\cdots \hat{\otimes}_{\mathcal{O}_{K}}\mathcal{O}_{C}}_{n-times})^{log}$
\end{center}
in the category of $p$-complete qfs-log rings.
\begin{rmk}\label{rmk-prod}
Consider the $p$-completed tensor product $\mathcal{O}_{C}\hat{\otimes}_{\mathcal{O}_{K}} \mathcal{O}_{C}$. It may be endowed with the pre-log structure $P:=\Q_{\geq0}\oplus_{\N}\Q_{\geq 0}\to \mathcal{O}_{C}\hat{\otimes}_{\mathcal{O}_{K}} \mathcal{O}_{C}$, induced by the prelog structure on $\mathcal{O}_{C}$ on each factor. The group completion $P^{gp}$ is then the abelian group $\Q\oplus_{\Z} \Q=\cong \Q/\Z\times \Q$. The monoid $P$ is not saturated: for example the element $x=(\frac{1}{3}, 0)-(0, \frac{1}{3})\in \Q\oplus_{\Z} \Q$ is not contained in $P$, but $3x=0\in P$. Indeed, it is easy to check that $P^{sat}=\Q/\Z\times \Q_{\geq 0}$. The coproduct $\mathcal{O}_{C}^{(2), log}$ in the category of $p$-complete qfs pre-log rings is thus given by $P^{sat}\to (\mathcal{O}_{C}\hat{\otimes}_{\mathcal{O}_{K}} \mathcal{O}_{C})\hat{\otimes}_{\mathbb{Z}[P]}\Z[P^{sat}]$.\\
In general, for $n\geq 1$, the coproduct $\mathcal{O}_{C}^{(n), log}$ in the category of $p$-complete qfs pre-log rings is given by $(\underbrace{\mathcal{O}_{C}\hat{\otimes}_{\mathcal{O}_{K}}\cdots \hat{\otimes}_{\mathcal{O}_{K}}\mathcal{O}_{C}}_{n-times})\hat{\otimes}_{\Z[P]}\Z[P^{sat}]$, where $P=\underbrace{\Q_{\geq 0}\oplus_{\N}\cdots \oplus_{\N}\Q_{\geq 0}}_{n-times}$, and $P^{sat}$ is seen to be isomorphic to $(\Q/\Z)^{n-1}\times \Q_{\geq 0}$.
\end{rmk}

\begin{lem}
The ring underlying the log product $\mathcal{O}_{C}^{(n),log}$ is quasi-regular semiperfectoid.
\end{lem}
\begin{proof}
We give the proof for $n=2$. The general case then follows by induction. In the category of $p$-complete pre-log rings, the coproduct is given by the usual tensor product $\mathcal{O}_{C}\hat{\otimes}_{\mathcal{O}_{K}} \mathcal{O}_{C}$ with pre-log structure given by the map $P:=\Q_{\geq0}\oplus_{\N}\Q_{\geq 0}\to \mathcal{O}_{C}\hat{\otimes}_{\mathcal{O}_{K}} \mathcal{O}_{C}$, induced by the prelog structure on $\mathcal{O}_{C}$ on each factor. The group completion $P^{gp}$ is then the abelian group $\Q\oplus_{\Z} \Q\cong \Q/\Z\times \Q$. By the previous remark, the saturation $P^{sat}$ is given by $\Q_{\geq 0}/\N\times \Q_{\geq 0}=\Q/\Z\times \Q_{\geq 0}$.\\
Let us write $R:=(\mathcal{O}_{C}\hat{\otimes}_{\mathcal{O}_{K}} \mathcal{O}_{C})\hat{\otimes}_{\mathbb{Z}[P]}\Z[P^{sat}]$. We need to prove that:
\begin{enumerate}
\item $R$ is quasi-syntomic.
\item There exists a map $S\to R$, with $S$ perfectoid.
\item $R/p$ is semiperfect.
\end{enumerate}
Condition (2) is obviously satisfied, as $\mathcal{O}_{C}$ maps to $R$. For (3), note that $\Q_{\geq 0}\times \Q/\Z$ is divisible, so that $\F_{p}[\Q_{\geq 0}\times \Q/\Z]$ is semiperfect. Since $\mathcal{O}_{C}\otimes_{\mathcal{O}_{K}}\mathcal{O}_{C}/p$ is semiperfect, we then also get that $R/p=\mathcal{O}_{C}\otimes_{\mathcal{O}_{K}}\mathcal{O}_{C}/p\otimes_{\F_{p}[P]}\F_{p}[P^{sat}]$ is semiperfect.\\
We are left with proving (1). For this, we first choose a compatible system $\{\pi^{\frac{1}{n}}\}_{n\in \N}$ of roots of the uniformizer $\pi \in \mathcal{O}_{K}$ and we define $L$ to be the $p$-adic completion of the field $\varinjlim_{n} K(\pi^{^{\frac{1}{n}}})$. Denote by $\mathcal{O}_{L}$ its ring of integers. We can then write $\mathcal{O}_{L}=(\mathcal{O}_{K}\otimes_{\Z[\N]}\Z[\Q_{\geq 0}])^{\wedge}$ and we endow $\mathcal{O}_{L}$ with the pre-log structure given by $\Q_{\geq 0}\to \mathcal{O}_{L}$. The extension $\mathcal{O}_{L}\to \mathcal{O}_{C}$ is then a strict map of pre-log rings and we get 
\begin{align*}
\mathcal{O}_{C}\hat{\otimes}^{log}_{\mathcal{O}_{K}}\mathcal{O}_{C} & = & ((\mathcal{O}_{C}\otimes_{\mathcal{O}_{L}}\mathcal{O}_{L})\otimes^{log}_{\mathcal{O}_{K}}(\mathcal{O}_{C}\otimes_{\mathcal{O}_{L}}\mathcal{O}_{L}))^{\wedge}\\
& = & (\mathcal{O}_{C}\otimes_{\mathcal{O}_{L}}(\mathcal{O}_{L}\otimes^{log}_{\mathcal{O}_{K}}\mathcal{O}_{L})\otimes_{\mathcal{O}_{L}}\mathcal{O}_{C})^{\wedge}\\
& =& (\mathcal{O}_{C}\otimes_{\mathcal{O}_{L}}(\mathcal{O}_{L}\otimes_{\mathcal{O}_{K}}\Z[\mu_{\infty}])\otimes_{\mathcal{O}_{L}}\mathcal{O}_{C})^{\wedge}
\end{align*}
Now, since $\mathcal{O}_{C}\otimes_{\mathcal{O}_{L}}(\mathcal{O}_{L}\otimes_{\mathcal{O}_{K}}\Z[\mu_{\infty}])\otimes_{\mathcal{O}_{L}}\mathcal{O}_{C}$ is the direct limit of syntomic $\Z_{p}$-algebras, we find that $\mathcal{O}_{C}\hat{\otimes}^{log}_{\mathcal{O}_{K}}\mathcal{O}_{C}$ is quasi-syntomic.
\end{proof}
\begin{prop}\label{initial}
Let $n\geq 0$. The log-prismatic site of $\mathcal{O}_{C}^{(n),log}$ has an initial object $(\prism_{\mathcal{O}_{C}^{(n),log}}, I,  M_{\prism_{\mathcal{O}_{C}^{(n),log}}})$.
\end{prop}
\begin{proof}
Since $\mathcal{O}_{C}^{(n),log}$ is qrsp, the ordinary prismatic site admits a final object $(\prism_{\mathcal{O}_{C}^{(n),log}}, I)$. 
By \cref{strictness}, the morphism $\mathcal{O}_{C}\to \mathcal{O}_{C}^{(n),log}$, mapping $\mathcal{O}_{C}$ to the first factor, is strict, i.e. the log structure on $\mathcal{O}_{C}\to \mathcal{O}_{C}^{(n),log}$ is induced by $\mathcal{O}_{C}\backslash \{0\} \hookrightarrow \mathcal{O}_{C}\to \mathcal{O}_{C}^{(n),log}$.\\
By universality, the morphism $\mathcal{O}_{C}\to \mathcal{O}_{C}^{(n),log}$ lifts to a morphism $A_{inf}\to \prism_{\mathcal{O}_{C}^{(n),log}}$. We then endow $\prism_{\mathcal{O}_{C}^{(n),log}}$ with the log structure $M_{\prism_{\mathcal{O}_{C}^{(n),log}}}$ induced by 
$\mathcal{O}_{C}^{\flat}\backslash \{ 0\} \to A_{inf}\to \prism_{\mathcal{O}_{C}^{(n),log}}$.\\
We claim that $(\prism_{\mathcal{O}_{C}^{(n),log}}, I, M_{\prism_{\mathcal{O}_{C}^{(n),log}}}, \delta_{log}\equiv 0)$ is the final object of the log-prismatic site.\\
Let $(B, J, M_{B}, \delta_{log})$ be an object in the log prismatic site of $\mathcal{O}_{C}^{(n),log}$. 
The map $\mathcal{O}_{C}^{(n),log}\to B/J$ lifts uniquely to a map of prisms $\prism_{\mathcal{O}_{C}^{(n),log}}\to B$. By Lemma \ref{lem-perf-log}, there exists a unique map of monoids $M_{A_{inf}}\to M_{B}$, such that the composition $A_{inf}\to \prism_{\mathcal{O}_{C}^{(n),log}}\to B$, becomes a map of log prisms. But since the first morphism is strict, we also get that $\prism_{\mathcal{O}_{C}^{(n),log}}\to B$ can be uniquely extended to a morphism of log rings.\\
The compatibility with the $\delta_{log}$-structures follows by the same arguments as in the proof of \cref{cor-perf}.
\end{proof}
\begin{rmk}\label{copr}
It is easy to see that $\prism_{\mathcal{O}_{C}^{(n),log}}$ identifies with the $n$-fold self coproduct of $A_{inf}$ in $(\mathcal{O}_{K}, M_{\mathcal{O}_{K}})_{\prism, \mr{kfl}}$.
\end{rmk}
\section{Log-prismatic Dieudonn\'e crystals}
We will now introduce log-prismatic Dieudonn\'e modules. 
\begin{defn}
Let $(R, M_{R})$ be a bounded $p$-complete log ring. A log-prismatic Dieudonn\'e module over $R$ is given by a pair $(\mathcal{M}, \Phi_{\mathcal{M}})$, where $\mathcal{M}$ is a locally free $\mathcal{O}_{\prism_{\mr{kfl}}}$-module and $\Phi_{\mathcal{M}}:\mathcal{M}\to \mathcal{M}$ is a $\phi_{\mathcal{O}_{\prism_{\mr{kfl}}}}$-linear map, such that the linearization
\[\phi^{*}\mathcal{M}\to \mathcal{M}\]
has $\mathcal{I}$-torsion cokernel.\\
We denote the category of log-prismatic Dieudonn\'e modules by $DM((R, M_{R})_{\prism, \mr{kfl}})$.
\end{defn}

\begin{prop}\label{LDM}
Let $R=\mathcal{O}_{C}^{(n), \mr{log}}$, for some $n\geq 0$. Then the evaluation functor 
\[DM((R, M_{R})_{\prism, \mr{kfl}})\to DM(\prism_{R})\]
is an equivalence of categories. Here $DM(\prism_{R})$ denotes the category of admissible Dieudonn\'e modules over the ring $\prism_{R}$ (see \cite[Def. 4.1.15]{ALB}).
\end{prop}
\begin{proof}
Over $R$, by \cref{initial}, we have the initial log prism $\prism_{R}$, so locally free log-prismatic crystals correspond to locally free sheaves on $(\prism_{R})_{\mr{kfl}}$. But by \cref{strictness} any Kummer log flat map of log-prism $\prism_{R}\to A$ is in fact strict, so that $(\prism_{R})_{\mr{kfl}}\cong (\prism_{R})_{\mr{fl}}$. But the category of locally free sheaves on $(\prism_{R})_{\mr{fl}}$ identifies with the category of projective modules over $\prism_{R}$.
\end{proof}
In particular, we see that any log prismatic Dieudonn\'e crystal over $\mathcal{O}_{C}$ is in fact a classical prismatic Dieudonn\'e crystal.\\[0.5cm]

We denote by $\prism_{\mathcal{O}_{C}}^{(\bullet), \mr{log}}$ the cosimplicial log-prism obtained by setting $\prism_{\mathcal{O}_{C}}^{(n), \mr{log}}=\prism_{\mathcal{O}_{C}^{(n), \mr{log}}}$, endowed with its canonical log-structure. We also further denote by 
$\begin{tikzcd}
    p_{1}, p_{2}:\prism_{\mathcal{O}_{C}} \arrow[r, shift left=0.5ex] \arrow[r, shift right=0.5ex, swap] & \prism_{\mathcal{O}_{C}^{(1), \mr{log}}}
    \end{tikzcd}$
the maps induced by the canonical maps $\begin{tikzcd}
    \mathcal{O}_{C} \arrow[r, shift left=0.5ex] \arrow[r, shift right=0.5ex, swap] & \mathcal{O}_{C}^{(1), \mr{log}}.
    \end{tikzcd}$

\begin{defn}
We denote by $DM(\prism_{\mathcal{O}_{C}}^{(\bullet), \mr{log}})$ the category consisting of triples $(M, \Phi_{M}, \epsilon)$, where
\begin{itemize}
\item $(M, \Phi_{M})\in DM(\prism_{\mathcal{O}_{C}})$.
\item $\epsilon:p_{1}^{*}(M, \Phi_{M})\xrightarrow{\cong}p_{2}^{*}(M, \Phi_{M})$ is an isomorphism, which satisfies the obvious cocycle condition over $\prism_{\mathcal{O}_{C}}^{(2), \mr{log}}$.
\end{itemize}
\end{defn}
\begin{prop}\label{DM-descent}
The evaluation functor induces an equivalence of categories
\[DM((\mathcal{O}_{K}, M_{\mathcal{O}_{K}}))_{\prism, \mr{kfl}})\longleftrightarrow DM(\prism_{\mathcal{O}_{C}}^{(\bullet), \mr{log}}).\]
\end{prop}
\begin{proof}
This follows from \cref{LDM}, using that $\prism_{\mathcal{O}_{C}}$ is weakly initial.
\end{proof}
We will now introduce the log-prismatic Dieudonn\'e functor. Let $G\in (p-div/\mathcal{O}_{K})^{\mr{log}}$ be a log $p$-divisible group. Then $G$ is a sheaf on $(\mathcal{O}_{K})_{kqsyn}$, since any $H\in (fin/(\mathcal{O}_{K}/p^{n}))_{f}$, for any $n\geq 1$, is a sheaf for the Kummer log-flat topology. Recall from \cref{topoi-diagram} that we have a funtor $\mu:(\mathcal{O}_{K}, M_{\mathcal{O}_{K}})^{\sim}_{\prism, \mr{kfl}}\to (\mathcal{O}_{K})_{kqsyn}^{\sim}$, which admits a left adjoint $\mu^{\natural}$ (see \cref{left-adjoint}).
\begin{defn}
Let $G$ be a log $p$-divisible group over $\mathcal{O}_{K}$. The log-prismatic Dieudonn\'e functor
\[\mathcal{M}_{\prism}^{\mr{log}}:(p-div/\mathcal{O}_{K})^{(\mr{\log})}\longrightarrow DM((\mathcal{O}_{K}, M_{\mathcal{O}_{K}})_{\prism, \mr{kfl}})\]
is defined by setting 
\[\mathcal{M}^{\mr{log}}_{\prism}(G):=\mathcal{E}xt^{1}_{(\mathcal{O}_{K}, M_{\mathcal{O}_{K}})_{\prism}}(\mu^{\natural}G, \mathcal{O}_{\prism, \mr{kfl}}).\]
\end{defn}
\begin{rmk}
In contrast to \cite{ALB}, we define the log Dieudonn\'e module to be a sheaf on the log-prismatic site, instead of the Kummer log quasisyntomic site. The reason is that the log quasisyntomic site seems to be harder to handle than the quasisyntomic site. This slight adjustment of terminology is harmless, by the results in loc.cit.: If $G$ is a classical $p$-divisible group, endowed with the log structure induced by $M_{\mathcal{O}_{K}}$, the sheaf $\mathcal{M}^{\mr{log}}_{\prism}(G)$ identifies with $\epsilon^{*}v^{*}\mathcal{M}_{\prism}(G)$. Here $\epsilon:(\mathcal{O}_{K}, M_{\mathcal{O}_{K}})^{\sim}_{\prism, \mr{kfl}}\to (\mathcal{O}_{K})^{\sim}_{\prism}$ is the projection from the log-prismatic topos to the prismatic topos, $v^{*}:(\mathcal{O}_{K})^{\sim}_{qsyn}\to (\mathcal{O}_{K})^{\sim}_{\prism}$ is the functor defined in \cite[Prop. 4.1.4]{ALB} and $\mathcal{M}_{\prism}(G)$ is the classical prismatic Dieudonn\'e module (see \cite[Def. 4.2.1]{ALB} and also \cite[Lem. 4.2.4]{ALB}). This gives an exact fully faithful functor from the category of $p$-divisible groups to the category $DM((\mathcal{O}_{K}, M_{\mathcal{O}_{K}})_{\prism})$ by \cite[Prop. 4.1.4. and Prop. 4.6.8.]{ALB}.
\end{rmk}

\begin{thm}
The log-prismatic Dieudonn\'e functor
\[\mathcal{M}_{\prism}^{\mr{log}}:(p-div/\mathcal{O}_{K})^{(\mr{\log})}\longrightarrow DM((\mathcal{O}_{K}, M_{\mathcal{O}_{K}})_{\prism, \mr{kfl}})\]
is an antiequivalence of categories.
\end{thm}
\begin{proof}
By \cref{DM-descent}, the category $DM((\mathcal{O}_{K}, M_{\mathcal{O}_{K}})_{\prism, \mr{kfl}})$ is equivalent to the category $DM(\prism_{\mathcal{O}_{C}}^{(\bullet), \mr{log}})$ Dieudonn\'e modules over $\prism_{\mathcal{O}_{C}}$ with logarithmic descent data. Similarly, as every log $p$-divisible group over $\mathcal{O}_{C}$ is classical (this follows from \cref{strictness}), the category of log $p$-divisible groups is equivalent to the category of $p$-divisible groups endowed with a descent datum over $\mathcal{O}_{C}^{(1), \mr{log}}$. The theorem thus follows from the classical case (\cite[Thm. 4.6.10.]{ALB}), using that the log products $\mathcal{O}_{C}^{(n), \mr{log}}$ are quasi-regular semiperfectoid.
\end{proof}

We will now extend the discussion on duality from \cite[\S 4]{ALB} to the log-case, using a slightly different terminology. Recall from \cite[Ex. 4.5]{BS-pris-cris} that on the prismatic site there is a Breuil-Kisin twist $\mathcal{O}_{\prism}\{-1\}$, which may informally be defined as $\bigotimes_{n\geq 0}(\phi^{n})^{*}\mathcal{I}^{-1}$. This is a Dieudonn\'e crystal with Frobenius map given by $\phi^{*}\mathcal{O}_{\prism}\{-1\}=\mathcal{I} \mathcal{O}_{\prism}\{-1\}\hookrightarrow \mathcal{O}_{\prism}\{-1\}$, which corresponds to the $p$-divisible group $\mu_{p^{\infty}}$. \\
Let $(\mathcal{M}, \Phi_{\mathcal{M}})\in DM((\mathcal{O}_{K}, M_{\mathcal{O}_{K}})_{\prism, \mr{kfl}})$ be a log-Dieudonn\'e module. Then we define the Cartier-dual $(\mathcal{M}^{\vee}, \Phi_{\mathcal{M}^{\vee}})$ of $(\mathcal{M}, \Phi_{\mathcal{M}})$ to be the log Dieudonn\'e module, defined by $\mathcal{M}^{\vee}:=\mathcal{H}om_{\mathcal{O}_{\prism}}(\mathcal{M}, \mathcal{O}_{\prism}\{-1\})$ and where the Frobenius $\Phi_{\mathcal{M}^{\vee}}$ is given by
\begin{align*}
\phi^{*}\mathcal{M}^{\vee} &= &\mathcal{H}om(\phi^{*}\mathcal{M}, \phi^{*}\mathcal{O}\{-1\})=\mathcal{H}om(\phi^{*}\mathcal{M}, \mathcal{I}\mathcal{O}\{-1\})\\& \longrightarrow & \mathcal{H}om(\mathcal{I}\mathcal{M}, \mathcal{I}\mathcal{O}\{-1\})=\mathcal{H}om(\mathcal{M}, \mathcal{O}\{-1\})=\mathcal{M}^{\vee},
\end{align*}
where the arrow in the middle is defined by realizing $\phi^{*}\mathcal{M}$ as a subsheaf of $\mathcal{M}$, using $\Phi_{\mathcal{M}}$ and then using that $\mathcal{M}/\phi^{*}\mathcal{M}$ is $\mathcal{I}$-torsion.\\
There is an obvious duality pairing
\[\mathcal{M}\otimes \mathcal{M}^{\vee}\to \mathcal{O}_{\prism}\{-1\},\]
compatible with the Frobenius maps.\\
The (log-)Dieudonn\'e functor is compatible with Cartier duality:

\begin{prop}
Let $G\in (p-div/\mathcal{O}_{K})^{(\mr{\log})}$ and denote by $G^{\vee}$ its Cartier-dual. Then
there is a canonical isomorphism
\[\mathcal{M}_{\prism}(G)^{\vee}\xrightarrow{\cong} \mathcal{M}_{\prism}(G^{\vee}).\]
\end{prop}
\begin{proof}
This follows by descent, using the compatibility in the classical case (\cite[Prop. 4.6.9]{ALB}).
\end{proof}

\subsection{Log-Dieudonn\'e modules on the strict log prismatic site}
Recall that there is a natural projection of sites $\epsilon_{\prism}:(\mathcal{O}_{K}, M_{\mathcal{O}_{K}})_{\prism, \mr{kfl}}\to (\mathcal{O}_{K}, M_{\mathcal{O}_{K}})_{\prism}$. The induced pullback functor $\epsilon_{\prism}^{*}$ on locally free crystals is fully faithful and thus also induces a full embedding $DM((\mathcal{O}_{K}, M_{\mathcal{O}_{K}})_{\prism})\hookrightarrow DM((\mathcal{O}_{K}, M_{\mathcal{O}_{K}})_{\prism, \mr{kfl}})$.

\begin{lem}
Let $G$ be a log $p$-divisible group, which is contained in $(p-div/\mathcal{O}_{K})^{(\mr{\log})}_{d}$. Then $\mathcal{M}_{\prism}^{\mr{log}}(G)$ is contained in $DM((\mathcal{O}_{K}, M_{\mathcal{O}_{K}})_{\prism})$.
\end{lem}
\begin{proof}
Since $G\in (p-div/\mathcal{O}_{K})^{(\mr{\log})}_{d}$, by a result of Kato, there exists an exact sequence
\[0\to G'\to G\to G'' \to 0\]
where $G'$ and $G''$ are classical $p$-divisible groups (this is in fact the connected \'etale sequence of $G$, see \cite[\S 2.6]{k-ld}). Let $(A, I, M_{A})$ be a log-prism over $\mathcal{O}_{K}$. The restriction $\mathcal{M}_{\prism}^{\mr{log}}(G)\vert_{(A, I, M_{A})}$ to $(\mathcal{O}_{K}, M_{\mathcal{O}_{K}})_{\prism, \mr{kfl}}/(A, I, M_{A})$ then corresponds to a locally free sheaf on $Spf(A)_{kfl}$. Since the log-Dieudonn\'e functor is exact, there is an exact sequence
\[0\to \mathcal{M}_{\prism}^{\mr{log}}(G')\vert_{(A, I, M_{A})}\to \mathcal{M}_{\prism}^{\mr{log}}(G)\vert_{(A, I, M_{A})}\to \mathcal{M}_{\prism}^{\mr{log}}(G'')\vert_{(A, I, M_{A})}\to 0.\]
But, since $G'$ and $G''$ are classical, $\mathcal{M}_{\prism}^{\mr{log}}(G')\vert_{(A, I, M_{A})}$ and $\mathcal{M}_{\prism}^{\mr{log}}(G'')\vert_{(A, I, M_{A})}$ are in fact locally free in the strict flat topology $Spf(A)_{fl}$. It then follows by \cite[Lem. 7.2.]{k-ld} that $\mathcal{M}_{\prism}^{\mr{log}}(G)\vert_{(A, I, M_{A})}$ is locally free on $Spf(A)_{fl}$ as well. But this means that $\mathcal{M}_{\prism}^{\mr{log}}(G)\in DM((\mathcal{O}_{K}, M_{\mathcal{O}_{K}})_{\prism})$. 
\end{proof}
We wish to show that every object in $DM((\mathcal{O}_{K}, M_{\mathcal{O}_{K}})_{\prism})$ comes from a log $p$-divisible group in $(p-div/\mathcal{O}_{K})^{(\mr{\log})}_{d}$. The following contructions are inspired by \cite[(1.1.16)]{Kis-moduli}.
\begin{defn}
We call a log-Dieudonn\'e module $(\mathcal{M}, \Phi_{\mathcal{M}})$ multiplicative, if $\phi_{\prism}^{*}\mathcal{M}\to \mathcal{M}$ is an isomorphism.\\
We call $(\mathcal{M}, \Phi_{\mathcal{M}})$ \'etale, if the image of $\phi_{\prism}^{*}\mathcal{M}\to \mathcal{M}$ is equal to $\mathcal{I}\mathcal{M}$.
\end{defn}
As in the case of minuscule Breuil-Kisin modules, it is easy to see that if $\mathcal{M}$ is a multiplicative (resp. \'etale) log-Dieudonn\'e module then $\mathcal{M}^{\vee}$ is \'etale (resp. multiplicative). It turns out that \'etale and flat multiplicative log Dieudonn\'e modules which lie in $DM((\mathcal{O}_{K}, M_{\mathcal{O}_{K}})_{\prism})$ are classical.
\begin{lem}
Let $(\mathcal{M}, \Phi_{\mathcal{M}})\in DM((\mathcal{O}_{K}, M_{\mathcal{O}_{K}})_{\prism})$ be  a multiplicative or \'etale log-Dieudonn\'e module. Then $(\mathcal{M}, \Phi_{\mathcal{M}})$ is classical. 
\end{lem}
\begin{proof}
We prove the claim for multiplicative $(\mathcal{M}, \Phi_{\mathcal{M}})$ - the \'etale case then follows using Cartier duality. Let $\epsilon:(\mathcal{O}_{K}, M_{\mathcal{O}_{K}})_{\prism}^{\sim}\to (\mathcal{O}_{K})_{\prism}^{\sim}$ be the natural projection of topoi. We need to show that $\epsilon^{*}\epsilon_{*}(\mathcal{M}, \Phi_{\mathcal{M}})\to (\mathcal{M}, \Phi_{\mathcal{M}})$ is an isomorphism. For this, we first claim that $(\mathcal{M}, \Phi_{\mathcal{M}})$ is associated to a $\Z_p$-local system over $\mathcal{O}_{K}$. More precisely, for $n\geq 1$, consider $(\mathcal{M}/p^{n}, \Phi_{\mathcal{M}}/p^{n})$. It corresponds to an $\mathfrak{S}$-module $M/p^{n}$ with a bijective Frobenius $\Phi_{M}/p^{n}$, with descent datum over $\mathfrak{S}^{(1), \mr{log}}$. The Frobenius invariants then define an \'etale finite locally free $\Z/p^{n}$-module $\mathbb{L}$ on $Spf(\mathfrak{S})$, such that $\tilde{M}/p^{n}=\mathcal{L}\otimes \mathcal{O}_{\mathfrak{S}}$, where $\tilde{M}$ denotes the coherent sheaf on the \'etale site of $Spf(\mathfrak{S})$ associated to $M$. To see this, use the faithfully flat map $\mathfrak{S}\to A_{inf}$. Over $A_{inf}$ one can then construct $\mathbb{L}$, using Artin-Schreier-Witt theory. The local system then descends to $\mathfrak{S}$, using fpqc descent of finite \'etale algebras. The prismatic descent data for $(\mathcal{M},\Phi_{\mathcal{M}})$ is then also induced from an isomorphism $(p^{(\mr{log}}_{1})^{*}\mathbb{L}\cong (p^{(\mr{log}}_{2})^{*}\mathbb{L}$ over $Spf(\mathfrak{S}^{(1), \mr{log}})$. We claim that this isomorphism is induced from an isomorphism $p_{1}^{*}\mathbb{L}\cong p_{2}^{*}\mathbb{L}$ over $\mathfrak{S}^{(1)}$ (the self-product of $\mathfrak{S}$ in the classical prismatic site). Namely, since $\mathfrak{S}^{(1)}$ (resp. $\mathfrak{S}^{(1), \mr{log}}$) is $E(u)$-complete, the category of finite \'etale sheaves over $Spf(\mathfrak{S}^{(1)})$ (resp. over $Spf(\mathfrak{S}^{(1, \mr{log})})$) is equivalent to the category of finite \'etale sheaves over $Spf(\mathfrak{S}^{(1)}/E(u))$ (resp. over $Spf(\mathfrak{S}^{(1, \mr{log})}/E(u))$). But now $p_{1}$ and $p_{2}$ (resp. $p_{1}^{\mr{log}}$ and $p_{2}^{\mr{log}}$) both reduce modulo $E(u)$ to the structure map $Spf(\mathfrak{S}^{(1)}/E(u))\to Spf(\mathcal{O}_{K})$ (resp. $Spf(\mathfrak{S}^{(1, \mr{log})}/E(u))\to Spf(\mathcal{O}_{K})$). In particular, the restriction of $p_{1}^{*}\mathbb{L}$ and $p_{2}^{*}\mathbb{L}$ to $Spf(\mathfrak{S}^{(1)}/E(u))$ both coincide with the pullback of $\mathbb{L}\vert_{Spf(\mathcal{O}_{K})}$ along $Spf(\mathfrak{S}^{(1)}/E(u))\to Spf(\mathfrak{S}^{(1)})$. We therefore see that there is a unique isomorphism $p_{1}^{*}\mathbb{L}\cong p_{2}^{*}\mathbb{L}$, which induces the isomorphism $(p^{(\mr{log}}_{1})^{*}\mathbb{L}\cong (p^{(\mr{log}}_{2})^{*}\mathbb{L}$.\\
Now the above constructed isomorphism $(p^{(\mr{log}}_{1})^{*}\mathbb{L}\cong (p^{(\mr{log}}_{2})^{*}\mathbb{L}$ induces a descent datum for $(M, \Phi_{M})$ over $\mathfrak{S}^{(1)}$, which induces the descent datum over $\mathfrak{S}^{(1, \mr{log})}$. This shows that $(\mathcal{M}, \Phi_{\mathcal{M}})$ indeed is classical, i.e. lies in the essential image of $\epsilon^{*}$. 
\end{proof}
The following lemma is a consequence of \cite[Prop. 1.2.11]{Kis-moduli}.
\begin{lem}
Let $(\mathcal{M}, \Phi_{\mathcal{M}})\in DM((\mathcal{O}_{K}, M_{\mathcal{O}_{K}})_{\prism})$ be a flat log-Dieudonn\'e module. Then there exists an exact sequence
\[0\to \mathcal{M}^{m}\to \mathcal{M}\to \mathcal{M}^{\mr{\acute{e}t}}\to 0\]
of flat log-Dieudonn\'e modules, where $\mathcal{M}^{m}$ is the maximal multiplicative sub-Dieudonn\'e module of $\mathcal{M}$ and $\mathcal{M}^{\mr{\acute{e}t}}$ is \'etale.
\end{lem}
\begin{proof}
We define $\mathcal{M}^{m}:=\bigcap_{n\geq 0}(\phi^{n})^{*}\mathcal{M} \subset \mathcal{M}$, where $(\phi^{n})^{*}\mathcal{M}$ is viewed as a subsheaf of $\mathcal{M}$ via
\[(\phi^{n})^{*}\mathcal{M}\xrightarrow{(\phi_{n})^{*}\Phi_{M}}(\phi^{n-1})^{*}\mathcal{M}\rightarrow \cdots \rightarrow \phi^{*}\mathcal{M}\xrightarrow{\Phi_{m}} \mathcal{M}.\]
We claim that $\mathcal{M}^{m}$ is a locally free crystal. Evaluating at $\mathfrak{S}$ and using \cite[Prop. 1.2.11]{Kis-moduli}, we see that $M^{m}:=\mathcal{M}^{m}(\mathfrak{S})$ and $M/M^{m}$ are finite projective $\mathfrak{S}$-modules. If $\mathfrak{S}\to A$ is a map of bounded log-prisms, we then see that $M^{m}\hat{\otimes}_{\mathfrak{S}}A\to \mathcal{M}(A)$ is injective. 
\begin{claim}
We have $M^{m}\hat{\otimes}_{\mathfrak{S}}A=\mathcal{M}^{m}(A)$ under the identification $M\hat{\otimes}_{\mathfrak{S}}A=\mathcal{M}(A)$.
\end{claim}
Namely, $\mathcal{M}^{m}(A)$ is identified with the kernel of the map $\mathcal{M}(A)\to \prod_{n\geq1}\mathcal{M}(A)/(\phi^{n})^{*}\mathcal{M}(A)$. Now $\mathcal{M}^{m}(A)=\bigcap_{n\geq 0}(\phi^{n})^{*}\mathcal{M}(A)$ is closed in $\mathcal{M}(A)$. Since $\mathcal{M}(A)$ is a projective $A$-module, it is bounded (as $A$ is bounded), hence $\mathcal{M}^{m}(A)$ is bounded as well, so it is classically $(p, E)$-complete. We clearly have $M^{m}\hat{\otimes}_{\mathfrak{S}}A\subset \mathcal{M}^{m}(A)$ and we want to show that this is an equality. $M^{m}\hat{\otimes}_{\mathfrak{S}}A$ identifies with the kernel of $M\hat{\otimes}A\to (\prod_{n}M/(\phi^{n})^{*}M)\hat{\otimes}A$. Now for any $k>0$, the ring $\mathfrak{S}/(p, E)^{k}$ is Artinian, so $Im(\bigcap_{n}(\phi^{n})^{*}M\to M/(p, E)^{k})=Im((\phi^{N_{k}})^{*}M\to M/(p, E)^{k})$, for some $N_{k}>>0$. From this we get that $coker(M^{m}\hat{\otimes}_{\mathfrak{S}}A\subset \mathcal{M}^{m}(A))/(p, E)^{k}=0$, for all $k>0$. Since $coker(M^{m}\hat{\otimes}_{\mathfrak{S}}A\subset \mathcal{M}^{m}(A))$ is $(p, E)$-complete, we see that $coker(M^{m}\hat{\otimes}_{\mathfrak{S}}A\subset \mathcal{M}^{m}(A))=0$, using Nakayama's lemma.\\
For a general map of log prisms $A\to B$ over $\mathcal{O}_{K}$, we get $\mathcal{M}^{m}(A)\hat{\otimes}_{A}B=\mathcal{M}^{m}(B)$ by descent, using that $\mathfrak{S}$ is weakly initial.
\end{proof}

\begin{prop}\label{flat-equiv}
The log-Dieudonn\'e functor restricts to an equivalence of categories
\[\mathcal{M}_{\prism}:(p-div/\mathcal{O}_{K})^{(\mr{\log})}_{d}\longleftrightarrow DM((\mathcal{O}_{K}, M_{\mathcal{O}_{K}})_{\prism}).\]
\end{prop}
\begin{proof}
Let $(\mathcal{M}, \Phi_{\mathcal{M}})\in DM((\mathcal{O}_{K}, M_{\mathcal{O}_{K}})_{\prism})$. Let $G$ be the log $p$-divisible group, such that $\mathcal{M}_{\prism}(G)=\mathcal{M}$. By the previous lemma, there is an exact sequence of log-Dieudonn\'e modules
\[0\to \mathcal{M}^{m}\to \mathcal{M}\to \mathcal{M}^{\acute{e}t}\to 0\]
such that $\mathcal{M}^{m}$ and $\mathcal{M}^{\acute{e}t}$ are classical. Using the exactness of the log Dieudonn\'e functor, we thus see that there is an exact sequence of log $p$-divisible groups (which of course identifies with the connected \'etale sequence)
\[0\to G'\to G\to G''\to 0\]
such that $G'$ and $G''$ are classical. But this in particular implies that $G$ is representable. By duality, one then also finds that $G^{\vee}$ is representable as well.
\end{proof}

\begin{rmk}
\begin{enumerate}
\item By a result of Kato (see \cite[Cor. 2.14]{WZ}), the category $(p-div/\mathcal{O}_{K})^{\mr{log}}_{d}$ is equivalent to the category of pairs $(G, N)$, where $G$ is a classical $p$-divisible group and $N:G^{\acute{e}t}(1)\to G^{\circ}$ is a morphism (the equivalence depends on the choice of a uniformizer). Using \cref{flat-equiv}, we see that we get a similar description for log Dieudonn\'e modules. Namely, there is an equivalence of categories (depending on $\pi \in \mathcal{O}_{K}$) between $DM((\mathcal{O}_{K}, M_{\mathcal{O}_{K}})_{\prism})$ and the category of pairs $((\mathcal{M}, \Phi_{\mathcal{M}}), N)$, where $(\mathcal{M}, \Phi_{\mathcal{M}})$ is a classical prismatic Dieudonn\'e module and $N:\mathcal{M}^{m}\to \mathcal{M}^{\acute{e}t}(-1)$ is a morphism of Dieudonn\'e modules.
\item By \cite[Thm. 1.1.3]{DL}, the category of flat log-Dieudonn\'e modules is equivalent, via the \'etale realization functor, to lattices in semistable $Gal(\bar{K}/K)$-representations with Hodge-Tate weights in $\{0, 1\}$. We thus see that \cref{flat-equiv} implies that the category $(p-div)_{d}^{\mr{log}}$ corresponds is anti-equivalent to the category of lattices in semistable $Gal(\bar{K}/K)$-representations with Hodge-Tate weights in $\{0, 1\}$. We remark that this has also already been proved in \cite{BWZ}, using other methods.
\item Assume that $p>2$. Let $\mathcal{S}$ be the Breuil-ring, which is the $p$-completed PD-hull of the kernel of $\mathfrak{S}\to \mathcal{O}_{K}$. It carries the PD-filtration $Fil_{\mathcal{S}}^{i}\mathcal{S}\subset \mathcal{S}$. Recall that a minuscule Breuil-module (or strongly divisible module) is given by a quadruple $(M, Fil^{1}M, \Phi_{M}, N)$, where
\begin{itemize}
\item $M$ is a finite free $\mathcal{S}$-module.
\item $Fil^{1}M\subset M$ is a submodule, such that $Fil^{1}_{\mathcal{S}}\cdot M\subset Fil^{1}M$.
\item $\Phi_{M}:M\to M$ is a Frobenius-linear map, such that the image of $Fil^{1}M$ generates $pM$.
\item A monodromy operator $N:M\to M$, satisfying $N(sm)=N_{\mathcal{S}}(s)m+sN(m)$, for all $s\in \mathcal{S}$ and $m\in M$, where $N_{\mathcal{S}}$ is the operator induced by $-u\frac{d}{du}$.
\end{itemize}
By a theorem of Breuil-Liu (see \cite[Theorem 2.3.5]{liu-conj}), the category $BM_{\mathcal{S}}^{1}$ of minuscule Breuil-modules is again equivalent to the category of lattices in semistable representations with Hodge-Tate weights in $\{0, 1\}$. So in particular it is also equivalent to the category $DM((\mathcal{O}_{K}, M_{\mathcal{O}_{K}})_{\prism})$ (and hence also anti-equivalent to the category $(p-div/\mathcal{O}_{K})^{(\mr{\log})}_{d}$).\\
The equivalence
\[DM((\mathcal{O}_{K}, M_{\mathcal{O}_{K}})_{\prism})\longleftrightarrow BM_{\mathcal{S}}^{1}\]
can in fact be constructed directly (i.e. without referring to \cite{liu-conj} and using \'etale realizations) by extending the approach in \cite{W} to log-objects.
\end{enumerate}
\end{rmk}

\bibliographystyle{plain}
\bibliography{bib-LD}

\end{document}